\documentclass[12pt,a4paper]{amsart}

\makeatletter
 \def\@textbottom{\vskip \z@ \@plus 582pt}
 \let\@texttop\relax
\makeatother

\usepackage{amsmath,amsfonts,amssymb,mathrsfs}
\usepackage[utf8]{inputenc}
\usepackage[T1]{fontenc}
\usepackage[english]{babel}
\usepackage{graphicx}
\usepackage{tabularx}
\usepackage{hyperref}
\usepackage{tikz}
\usepackage{cleveref}
\usetikzlibrary{arrows}
\usepackage{fancyhdr}

\allowdisplaybreaks

\newcommand{\nocontentsline}[3]{}
\newcommand{\tocless}[2]{\bgroup\let\addcontentsline=\nocontentsline#1{#2}\egroup}
\DeclareUnicodeCharacter{0177}{\^u}

\newtheorem{theorem}{Theorem}[section]

\newtheorem{lemma}[theorem]{Lemma}
\newtheorem{definition}[theorem]{Definition}

\newtheorem{proposition}[theorem]{Proposition}

\newcommand{\Z}{\mathbb{Z}}
\renewcommand{\P}{\mathbb{P}}

 \DeclareFontFamily{U}{wncy}{}
 \DeclareFontShape{U}{wncy}{m}{n}{<->wncyr10}{}
 \DeclareSymbolFont{mcy}{U}{wncy}{m}{n}
 \DeclareMathSymbol{\Sh}{\mathord}{mcy}{"58}

\begin{document}

\title[]{Rational curves on elliptic K3 surfaces} 
\author{Salim Tayou }
\date\today
\thanks{\textit{École Normale Supérieure, 45 rue d'Ulm, 75005 Paris, France}}
\thanks{\textit{Laboratoire de Mathématiques d'Orsay, Université Paris-Sud, Orsay, France}}
\thanks{E-mail adress: salim.tayou@math.u-psud.fr}

\begin{abstract}
We prove that any non-isotrivial elliptic K3 surface over an algebraically closed field $k$ of arbitrary characteristic contains infinitely many rational curves. In the case when $\mathrm{char}(k)\neq 2,3$, we prove this result for any elliptic K3 surface. When the characteristic of $k$ is zero, this result is due to the work of Bogomolov-Tschinkel and Hassett.
\end{abstract}

\maketitle

\section{Introduction}
Let $X$ be a K3 surface over an algebraically closed field $k$. In \cite[Corollary 3.28]{bt2}, Bogomolov and Tschinkel prove that when the characteristic of $k$ is zero and $X$ admits a non-isotrivial elliptic fibration then $X$ contains infinitely many rational curves. Later, Hassett in \cite[Section 9]{hassett} handled the general case of arbitrary elliptic complex K3 surfaces. In this note, we extend the above results to the case where $k$ has positive characteristic. 
\begin{theorem}\label{principal}
Let $X$ be an elliptic K3 surface over an algebraically closed field $k$. Then $X$ contains infinitely many rational curves in the following cases:
\begin{enumerate}
\item $X$ admits a non-isotrivial elliptic fibration;
\item $\mathrm{char}(k)\neq 2,3$. 
\end{enumerate}
\end{theorem}
In characteristic zero, this is the content of \cite[Corollary 3.28]{bt2} and \cite[Section 9]{hassett}. When $k$ has positive characteristic, the main ingredients in case $(1)$ are a result on the image of $\ell$-adic monodromy representations associated to non-isotrivial $1$-dimensional families of elliptic curves, see \Cref{independence}. The proof is inspired from \cite{bt2}, though we simplify some arguments presented there. The proof in case $(2)$ follows the arguments of Hassett in \cite[Section 9]{hassett}. This note is split into two parts. In the first section, some background on elliptic K3 surfaces is recalled. The main result is proved in the second section. 
\subsection{Acknowledgements}It is a great pleasure to thank François Charles for his discussions and valuable remarks about this note. Further thanks go to Frank Gounelas, Quentin Guignard, Max Lieblich and Gebhard Martin for valuable remarks and discussions on this text. This project has received funding from the European Research Council (ERC) under the European Union’s Horizon 2020 research and innovation programme (grant agreement No 715747).

\section{Background on elliptic K3 surfaces}
Let $k$ be an algebraically closed field of positive characteristic and $\mathbb{P}^{1}_{k}$ the projective line over $k$. We recall some facts about elliptic K3 surfaces. For a more comprehensive introduction, see \cite[Chapter 11]{huybrechts}. 
\bigskip

An elliptic K3 surface is a K3 surface $X$ which admits a surjective morphism $X\xrightarrow{\pi} \P^1_{k}$ whose generic fiber is a smooth integral curve of genus $1$. If moreover the morphism $\pi$ admits a section, then $X$ is said to be a Jacobian elliptic K3 surface. The fibration is said to be non-isotrivial if not all the smooth fibers are isomorphic. For Jacobian elliptic K3 surfaces, the latter condition is equivalent to the fact that the $j$-invariant of the generic fiber is not in $k$.

\subsection{Tate-Shafarevich group} Let $X\xrightarrow{\pi} \P^1_k$ be an elliptic K3 surface. For every integer $d\geq 0$, one can associate to $X$ an elliptic K3 surface $J^d(X)$ as follows. If $\eta$ denotes the generic point of $\P^1_{k}$, then the generic fiber $X_\eta$ over $k(\eta)$ is a smooth integral curve of genus $1$. Then one can associate to it a smooth curve of genus $1$, $Jac^{d}(X_\eta)$, which coarsly represents the étale sheafification of the functor $$\mathrm{Pic}^d:(\mathrm{Sch}/k(\eta))^{\circ}\rightarrow (\mathrm{Sets}),\,S\mapsto \mathrm{Pic}^d(X_\eta\times S)/\sim.$$ 
Then $J^d(X)\rightarrow \P^1_{k}$ is defined as the unique relatively minimal smooth model of  $Jac^d(X_\eta)$. For $d=0$, we denote it simply $J(X)$ and it is a Jacobian elliptic K3 surface, see\cite[Chap.11, Section 4.1]{huybrechts} or \cite[Thm. 5.3.1]{cossec} for more details. For every smooth fiber $X_t,\, t\in \P^1_k$, the fiber $J(X)_t$ is isomorphic to the Jacobian elliptic curve associated to $X_t$. Let $J(X)^{sm}\subset J(X)$ be the open set of $\pi$-smooth points, viewed as a smooth group scheme over $\P^1_k$. Then the open $\pi$-smooth locus $X^{sm}\rightarrow \P^1_k$ is a $J(X)^{sm}$-torsor over $\P^1_k$. Hence for an arbitrary Jacobian elliptic K3 surface $Y\rightarrow \mathbb{P}^1_k$, define the \textit{Tate-Shafarevich group} $\Sh(Y)$ as the set of isomorphism classes of $Y^{sm}$-torsors over $\P^1_k$. The group structure on $\Sh(Y)$ depends on the choice of the section, however the isomorphism class does not. 
 
\begin{proposition}[Chap.11, Section 5.2, 5.5(i), 5.6 \cite{huybrechts}]\label{brauer}
Let $X\rightarrow \P^1_k$ be a Jacobian elliptic K3 surface. The Tate-Shafarevich group $\Sh(X)$ is isomorphic to the Brauer group $\mathrm{Br}(X)$ of $X$ and we have an injective map $$\Sh(X)\hookrightarrow WC(X_\eta),$$
where $WC(X_\eta)$ is the Weil-Châtelet group of the generic fiber of $X\rightarrow \P^1_k$.
\end{proposition}
Recall that the Brauer group of $X$ is defined as the étale cohomology group $H^{2}(X,\mathbb{G}_m)$ and recall also that for an elliptic curve $E$ over a field $K$, the Weil-Châtelet group, denoted $WC(E)$, is defined as the set of isomorphism classes of torsors under $E$ over $K$, see \cite[Chapter 11, Section 5.1]{huybrechts}. 
\bigskip

For every positive integer $d$ and for every smooth fiber $X_t,\, t\in \P^1_k$, $J^d(X)_t$ is isomorphic to $\mathrm{Pic}^d(X_t)$.   
Moreover, one has an isomorphism 
\begin{center}
\begin{tikzpicture}[scale=0.75]

\node (X) at (0,0) {$X$};

\node (Xp) at (5,0) {${J}^1(X)$};
\node (P1) at (2.5,-3) {$\P^1_k$};

\draw[->] (X)--node[above] {$\sim$} (Xp);
\draw[->] (Xp)--node[right]{$\pi_{1}$}(P1);
\draw[->] (X)--node[left]{$\pi$}(P1);
\end{tikzpicture}
\end{center} and $J(J^d(X))\simeq J(X)$. In addition, the class $[J^d(X)]$ of $J^d(X)$ in $\mathrm{Br}(J(X))$ is equal to $d[X]$. 
\bigskip

For every integers $d,d'$, we have natural rational maps of algebraic varieties 
\begin{center}
\begin{tikzpicture}[scale=0.75]
\node (X) at (0,0) {$J^d(X)\times_{\P^1_k} J^{d'}(X)$};
\node (Xp) at (5,0) {$J^{d+d'}(X)$};
\node (P1) at (2.5,-3) {$\P^1_k$};

\draw[dashed,->] (X)--node[above] {} (Xp);
\draw[->] (Xp)--node[right]{}(P1);
\draw[->] (X)--node[left]{}(P1);
\end{tikzpicture}
\end{center}

For a positive integer $\ell$, the diagonal embedding $$ J^1(X)\rightarrow \underbrace{J^1(X)\times_{\mathbb{P}^{1}_{k}}\dots\times_{\mathbb{P}^{1}_{k}} J^1(X)}_{\ell\, \textrm{times}}$$ composed with the rational map above defines a rational map $\eta_\ell$ which fits into the following commutative diagram
\begin{center}\label{sum}
\begin{tikzpicture}[scale=0.75]
\node (X) at (0,0) {$J^1(X)$};
\node (Xp) at (5,0) {$J^\ell(X)$};
\node (P1) at (2.5,-3) {$\P^1_k$};

\draw[dashed,->] (X)--node[above] {$\eta_\ell$} (Xp);
\draw[->] (Xp)--node[right]{$\pi_{\ell}$}(P1);
\draw[->] (X)--node[left]{$\pi$}(P1);
\end{tikzpicture}
\end{center}
The map $\eta_\ell$ is defined over the smooth locus of $\pi$. 
\subsection{Rational curves}
Let $X$ be a K3 surface over $k$. A rational curve on $X$ is an integral closed subscheme $C$ of dimension $1$ and of geometric genus $0$. Recall the following existence result, attributed to Bogomolov and Mumford, with a refinement of Li and Liedtke (\cite[Theorem 2.1]{ll}). 
\begin{proposition}[Bogomolov-Mumford]\label{BM}
Let $L$ be a non-trivial effective line bundle on a K3 surface $X$ over $k$. Then $L$ is linearly equivalent to a sum of effective rational curves. 
\end{proposition}
\subsection{Relative effective Cartier divisors}
\begin{definition}
Let $X\rightarrow \P^1_k$ be an elliptic K3 surface. A relative effective Cartier divisor on $X/\P^1_k$ is a closed subscheme $\mathcal{M}$ on $X$ such that $\mathcal{M}\rightarrow \P^1_k$ is finite flat. If moreover $\mathcal{M}$ is irreducible, it is called a multisection. 
\end{definition}
Given an elliptic K3 surface $X$ and a multisection $\mathcal{M}$ on $X$, the map $\mathcal{M}\rightarrow \P^1_k$ is finite flat and its degree is by definition the degree of $\mathcal{M}$.
\bigskip

Let $X_0$ be a smooth fiber of $X\rightarrow \P^1_k$ over a point $0\in\mathbb{P}^{1}_{k}$. Then we have a map given by the intersection product $$\mathrm{Pic}(X)\xrightarrow{(X_0,\;)} \Z.$$
It sends any multisection to its degree. The image of the above map is a non-zero subgroup of $\Z$, of finite index. Denote by $d_{X}$ its index. It is called the degree of the elliptic fibration $X\rightarrow \P_k^1$. Remark that an elliptic fibration is Jacobian if and only if its degree is equal to one.
\begin{lemma} \label{lemma}
Let $X\rightarrow \mathbb{P}^{1}_{k}$ be an elliptic K3 surface.
\begin{enumerate}
\item The order of $[X]$ in $\mathrm{Br}(J(X))$ is equal to $d_{X}$. 
\item There exists a multisection of degree $d_{\mathcal{M}}=d_{X}$ which is a rational curve. 
\item There exists at least one multisection $\mathcal{M}$ such that $d_{\mathcal{M}}=d_{X}$ and which is moreover generically étale over $\P^1_k$.
\end{enumerate}
\end{lemma}
\begin{proof}[Proof]
For $(2)$, let $\mathcal{M}$ be a multisection of degree $d_{X}$. By \Cref{BM}, $\mathcal{M}$ is linearly equivalent to a sum of rational curves $\sum_{i}C_i$. Then there exists a unique curve $C_i$ which is horizontal and all the others are vertical. Then $C_i$ satisfies the desired properties.
 
For $(1)$, notice that $X_{\eta}$ is a torsor under the elliptic curve $J(X)_{\eta}$ and that $d_X$ is the index of $X_{\eta}$, i.e is the greatest common divisor of the degrees of residue fields of closed points of $X_{\eta}$ (see \cite[1]{lichtenbaum}). Since the order of $X_\eta$ in $WC(J(X)_{\eta})$ is equal to its index by \cite[Theorem 1]{lichtenbaum}, it implies that the order of $[X]$ is exactly $d_X$. By \cite[Section 5, Theorem 4]{lichtenbaum}\footnote{More precisely, see the proof given there.}, it is also equal to the minimal degree of residue fields of separable closed points. Hence there exists a closed separable point in $X_{\eta}$ of degree $d_X$. Taking its closure yields a separable multisection. This proves $(3)$.
\end{proof}


\subsection{Monodromy}
Let $X\xrightarrow{\pi} \P^1_k$ be an elliptic K3 surface. Let $U$ be the largest Zariski open subset of $\mathbb{P}^{1}_k$ over which the map $\pi$ is smooth. Thus $X_U\rightarrow U$ is a torsor under the smooth group scheme $J(X)_{U}\rightarrow U$. For $b\in U$ a closed point and $m$ prime to $p:=\mathrm{char}(k)$, the étale fundamental group $\pi_{1}^{\textrm{ét}}(U,b)$ of $U$ acts on the group of $m$-torsion points in $J(X)_b$ and defines a group morphism 
\begin{align*}
\rho:\pi^{\textrm{ét}}_{1}(U,b)\rightarrow \mathrm{Aut}\left(\lim_{\underset{\gcd(m,p)=1}{\longleftarrow}}J(X)_{b}[m]\right)=\prod_{\gcd(\ell,p)=1}\mathrm{Aut}(\mathrm{T}_{\ell}J(X)_{b}).
\end{align*}
This action preserves the Weil paring and factors as follows: 
\begin{align*}
\rho:\pi^{\textrm{ét}}_{1}(U,b)\rightarrow \prod_{\ell\wedge p=1}\mathrm{SL}(\mathrm{T}_{\ell}J(X)_{b}).
\end{align*}

For every prime $\ell$, we denote by $\rho_{\ell^\infty}$ the representation of $\pi_{1}^{\textrm{ét}}(U,b)$ on the Tate module $\mathrm{T}_{\ell}J(X_{b})$ and denote by $\rho_\ell$ its reduction modulo $\ell$. Then $\rho_{\ell^\infty}$ is simply the projection on the $\ell$-factor in the previous map. The monodromy group $\Gamma$ is the image of $\pi^{\textrm{ét}}_{1}(U,b)$ under $\rho$. The next result on the image of the monodromy group will be crucial in the proof of \Cref{principal}. 

\begin{proposition}[\cite{hall}]\label{independence}
If the elliptic fibration is not isotrivial, then there exists a constant $c(k)$ depending only on $k$, such that for every $\ell>c(k)$ the morphism $\rho_{\ell}$ is surjective. 
\end{proposition}
This is the content of \cite[Theorem 1.1]{hall} where the surjectivity is proven for the reduction modulo $\ell$, then one uses Lemma $2$ in \cite[IV-23]{serreladic}. Notice that in \cite[Theorem 1.1]{hall}, the base field is supposed to be finite but one can check that the proof given there works for perfect fields, as mentioned in the discussion after Theorem 1.1 in {\it loc.cit}. 

 
\section{Proof of Theorem \ref{principal}}\label{part2} 
If $X$ has Picard rank $\rho(X)$ at least $20$, then the automorphism group of $X$ is infinite and hence $X$ contains infinitely many rational curves, see \cite[Chap.13, Remark 1.6]{huybrechts} and \cite[Theorem 4.1]{bt2}. Hence we assume that $\rho(X) \leq 19$. 

The elliptic surface $X$ defines a class in the Tate-Shafarevich group $\Sh(J(X))$ of $J(X)$, which is isomorphic to the Brauer group $\mathrm{Br}(J(X))$ by  \Cref{brauer}. This class is a sum of two elements $\alpha_{p}+\alpha$, where $\alpha$ has torsion prime to $p$ and $\alpha_p$ is torsion of order $p^a$, for some integer $a$. Here $p$ is the characteristic of $k$. We will construct infinitely many multisections on $X$ which are rational curves and whose degrees tend to infinity. This will be enough to prove \Cref{principal}. Denote by $d_{X}$ the degree of $X$ and let $\ell$ be a prime number with residue $1 \pmod {p^a} $ and such that $\ell>\max(d_X,c(k))$, where $c(k)$ is given by \Cref{independence}. The prime to $p$ torsion part of $\mathrm{Br}(J(X))$ is a divisible group by \cite[Chap. 18, Example 1.5]{huybrechts}. 
The Kummer exact sequence and the assumption on the Picard rank ensures furthermore that it is not trivial (see formula (1.8) {\it loc. cit}).
We can thus find an elliptic K3 surface $\pi_\ell:X_\ell\rightarrow \mathbb{P}^{1}$ such that $J(X_{\ell})\simeq J(X)$, $\ell[X_{\ell},\pi_\ell]=[X,\pi]$ in $\mathrm{Br}(J(X))$ and $d_{X_\ell}=\ell d_{X}$. Take for instance $\alpha_{p}+\alpha_{\ell}$, where $\alpha_{\ell}$ is a non-trivial element in $\mathrm{Br}(J(X))$ which satisfies $\ell.\alpha_{\ell}=\alpha$. Hence $J^\ell(X_{\ell})\simeq X$ and we have a rational map defined at the end of section 2.1: 
\begin{center}
\begin{tikzpicture}[scale=0.75]
\node (X) at (5,0) {$X$};
\node (Xp) at (0,0) {$X_\ell$};
\node (P1) at (2.5,-3) {$\P^1_k$};

\draw[dashed,->] (Xp)--node[above] {$\eta_\ell$} (X);
\draw[->] (X)--node[right]{$\pi$}(P1);
\draw[->] (Xp)--node[left]{$\pi_{\ell}$}(P1);
\end{tikzpicture}
\end{center}
By \Cref{lemma}, $X_\ell$ contains a rational  multisection $\mathcal{M}_\ell$ of degree $d_{\mathcal{M}_\ell}=d_{X_\ell}=\ell d_{X}$. If the restriction of $\eta_\ell$ to $\mathcal{M}_\ell$ is isomorphic to its images above $\mathbb{P}_{k}^1$ then $\eta_\ell(\mathcal{M}_\ell)$ is a rational curve on $X$ of degree divisible by $\ell$ which is the desired result. Otherwise, since the multiplication by $\ell$ map is étale (by \cite[Théorème 2.5]{grothendieckfga}), there exists infinitely many closed points $b$ in the maximal open subset $U\subset \P^1_k$ where $\pi$ is smooth, $\mathcal{M}_{\ell,U}\rightarrow U$ is smooth and two distinct points $P_1$, $P_2$ in $X_{\ell,b}\cap \mathcal{M}_\ell$ such that $\ell.(P_1-P_2)=0$ in $J(X)_b$. Thus, the point $P_1-P_2$ is a $\ell$-primitive torsion point in $J(X)_b$. 
Let $J(X)_U[\ell]\rightarrow U$ be the relative effective Cartier divisor of $J(X)_U\rightarrow U$ of $\ell$-torsion points. 

Let $J(X)_{U,prim}[\ell]$ be the relative effective Cartier divisor of non-zero $\ell$-torsion points. Since $X_{\ell,U}$ is a $J(X)_U$-torsor over $U$, there is an induced map: 
\begin{align}\label{component}
J(X)_{U,prim}[\ell]\times\mathcal{M}_{\ell,U}\rightarrow X_{\ell,U}.
\end{align}
The closure of the image in $X_{\ell}$ is a curve of $X_{\ell}$ which intersects $\mathcal{M}_{\ell}$ infinitely many times by the non-injectivity of $\eta_\ell$. Hence $\mathcal{M}_{\ell}$ is isomorphic to an irreducible component of $J(X)_{U,prim}[\ell]\times_U\mathcal{M}_{\ell,U}$.

\subsection{Non-isotrivial case} 
For $\ell$ large enough, $J(X)_{U,prim}[\ell]$ is irreducible by \Cref{independence}. Hence via its first projection, the above map is surjective over $J(X)_{U,prim}[\ell]$. Since there are $\ell^2-1$ torsion points in each fiber of $J(X)_{U,prim}[\ell]$ over $U$, this implies $$d_{\mathcal{M}_{\ell}}=\ell d_{X}\geq \ell^2-1.$$ This is a contradiction by our assumption on $\ell$.

\subsection{Isotrivial case}
We assume now that the elliptic fibration $X\rightarrow \mathbb{P}_k^{1}$ is isotrivial. Then the elliptic fibration $J(X)\rightarrow \mathbb{P}_k^{1}$ is also isotrivial. If the characteristic of $k$ is different from $2$ and $3$, which will be assumed henceforth, then we can proceed following the lines of \cite[Section 9]{hassett}. The image of the étale fundamental group of $U$ by $\rho_\ell$ factors through the automorphism group of the geometric generic fiber of $J(X)\rightarrow \mathbb{P}_k^{1}$ which is cyclic of order $2,4$ or $6$, see \cite[III.10]{silverman}. Assume that the fibration $J(X)\rightarrow \mathbb{P}_k^{1}$ has $n_0$ degenerate fibers of type $I^{*}_0$, $n'_1$ degenerate fibers of type $I_a$, $a>0$, $n^{''}_1$ degenerate fibers of type $I^{*}_a$, $a>0$, $n_2$ fibers of type $II$ or $II^*$, $n_3$ fibers of type $III$ or $III^*$, and $n_4$ fibers of type $IV$ or $IV^*$. For the definition of the type of singularities of fibers, see \cite[Chapter 11, Section 1.3]{huybrechts}.

By \Cref{component}, $\mathcal{M}_{\ell,U}$ is an irreducible component of a principal homogeneous space under $J(X)_{U,prim}[\ell]$. Using Riemann-Hurwitz as in the proof of \cite[Theorem 9.9]{hassett} and noticing that the computations of the ramification contributions of degenerate fibers from \cite[Table 1, page 259]{hassett} hold for $\ell$ large enough, see \cite[Chapitre III, 17]{neron}, there exists $C>0$ such that $g(\mathcal{M}_\ell)\geq C .c(J)$ where $g(\mathcal{M}_\ell)$ is the geometric genus of $\mathcal{M}_\ell$ and 
$$c(J)=\frac{1}{2}n_0+n'_1+n^{''}_1+\frac{5}{6}n_2+\frac{3}{4}n_3+\frac{2}{3}n_4-2.$$
Since $\mathcal{M}_\ell$ is a rational curve, we infer that $c(J)\leq 0$. We use now the method of \cite[Proposition 9.6]{hassett} to classify K3 surfaces that satisfy the last condition. By Shioda-Tate formula \cite[Theorem 6.3]{schutt}), we have : 
\begin{align*}
\rho(X)=2+\sum_{s\in\mathbb{P}^{1}(k)}(r_s-1)+r(X)
\end{align*}
where $r_s$ denotes the number of irreducible components of a fiber $X_s$ for $s$ a closed point in $\mathbb{P}^{1}_k$ and $r(X)$ is the rank of the Mordell-Weil group of $J(X)$. On the other hand, the $\ell$-adic Euler formula (\cite[Theorem 1.1, Corollary 1.6]{dolgachev}\footnote{With the correct sign.}) implies that: 
\begin{align}\label{ladic}
24=\sum_{s\in\mathbb{P}^1(k)}\left[\chi(X_s)_\ell+\alpha_{s,\ell}\right]
\end{align}
where, for $s\in\mathbb{P}^{1}_k(k)$, $\chi(X_s)_\ell$ is the $\ell$-adic Euler characteristic of the fiber $X_s$ and $\alpha_{s,\ell}$ is its wild conductor defined in \cite[Section 1]{dolgachev}. Recall that $r_s=\chi(X_s)_\ell$ if the fiber $X_s$ has reduction type $I_a$ and otherwise $r_s=\chi(X_s)_\ell-1$. Since the characteristic of $k$ is different from $2$ and $3$, all the wild conductors above vanish.

Combining the two previous formulas, we get: 
\begin{align*}
\rho(X)&=2+\sum_{\underset{\textrm{of type $I_a$}}{s\in\mathbb{P}^{1}_k(k)}}(r_s-1)+\sum_{\underset{\textrm{not of type $I_a$}}{s\in\mathbb{P}_k^{1}(k)}}(r_s-2)+r(X)\\
&=26-{n'}_1-2N+r(X)
\end{align*}
where $N=n_0+{n''}_1+n_2+n_3+n_4$.
The assumption that $c(J)\leq 0$ implies that
\begin{align*}
18+r(X)+3n^{'}_1+2n^{''}_1+\frac{4}{3}n_2+n_3+\frac{2}{3}n_4\leq \rho(X).
\end{align*}
Hence either $X$ has Picard rank equal to $22$, or $\rho(X)\leq 20$ and thus $X$ is an element in the list given in \cite[Proposition 9.6]{hassett}. In all these cases, $X$ is either a Kummer surface or its automorphism group is infinite. In both cases, $X$ has infinitely many rational curves, see \cite[Corollary 4.3]{bt1} and \cite[Lemma 4.9]{bt2} for the second case. 
\subsection{Situation in characteristic 2 and 3}
When the characteristic of $k$ is equal to $2$ or $3$ and the elliptic fibration $X\rightarrow \mathbb{P}_k^{1}$ is isotrivial then the classification above must be modified to take into account the wild ramification factors in \Cref{ladic} which do not vanish in general, apart from special cases, see \cite[Section 4.6, Table 2]{schutt}. For example, we could have a K3 surface with a single cusp of conductor $24$ for which $c(J)=\frac{-7}{6}$ and $\rho(X)\geq 2$. It would be interesting to investigate these small rank situations. 
\bibliographystyle{alpha}
\bibliography{bibliographie}

\begin{thebibliography}{{Gro}62}

\bibitem[BT00]{bt2}
F.~A. Bogomolov and Yu. Tschinkel.
\newblock Density of rational points on elliptic {$K3$} surfaces.
\newblock {\em Asian J. Math.}, 4(2):351--368, 2000.

\bibitem[BT05]{bt1}
Fedor Bogomolov and Yuri Tschinkel.
\newblock Rational curves and points on {$K3$} surfaces.
\newblock {\em Amer. J. Math.}, 127(4):825--835, 2005.

\bibitem[CD89]{cossec}
Fran\c{c}ois~R. Cossec and Igor~V. Dolgachev.
\newblock {\em Enriques surfaces. {I}}, volume~76 of {\em Progress in
  Mathematics}.
\newblock Birkh\"auser Boston, Inc., Boston, MA, 1989.

\bibitem[CH05]{hall}
Alina~Carmen Cojocaru and Chris Hall.
\newblock Uniform results for {S}erre's theorem for elliptic curves.
\newblock {\em Int. Math. Res. Not.}, (50):3065--3080, 2005.

\bibitem[{Dol}72]{dolgachev}
I.~V. {Dolgachev}.
\newblock {Euler characteristic of a family of algebraic varieties.}
\newblock {\em {Mat. Sb., Nov. Ser.}}, 89:297--312, 1972.

\bibitem[{Gro}62]{grothendieckfga}
Alexander {Grothendieck}.
\newblock {Technique de descente et th\'eoremes d'existence en g\'eom\'etrie
  alg\'ebrique. VI: Les schemas de Picard. Propri\'et\'es g\'en\'erales.}
\newblock {Sem. Bourbaki 14(1961/62), No.236, 23 p. (1962).}, 1962.

\bibitem[Has03]{hassett}
Brendan Hassett.
\newblock Potential density of rational points on algebraic varieties.
\newblock In {\em Higher dimensional varieties and rational points ({B}udapest,
  2001)}, volume~12 of {\em Bolyai Soc. Math. Stud.}, pages 223--282. Springer,
  Berlin, 2003.

\bibitem[Huy16]{huybrechts}
Daniel Huybrechts.
\newblock {\em Lectures on {K}3 surfaces}, volume 158 of {\em Cambridge Studies
  in Advanced Mathematics}.
\newblock Cambridge University Press, Cambridge, 2016.

\bibitem[Lic68]{lichtenbaum}
Stephen Lichtenbaum.
\newblock The period-index problem for elliptic curves.
\newblock {\em Amer. J. Math.}, 90:1209--1223, 1968.

\bibitem[LL12]{ll}
Jun Li and Christian Liedtke.
\newblock Rational curves on {K}3 surfaces.
\newblock {\em Inventiones mathematicae}, 188(3):713--727, 2012.

\bibitem[N\'64]{neron}
Andr\'e N\'eron.
\newblock Mod\`eles minimaux des vari\'et\'es ab\'eliennes sur les corps locaux
  et globaux.
\newblock {\em Publications Math\'ematiques de l'IH\'ES}, 21:5--128, 1964.

\bibitem[Ser98]{serreladic}
Jean-Pierre Serre.
\newblock {\em Abelian {$l$}-adic representations and elliptic curves},
  volume~7 of {\em Research Notes in Mathematics}.
\newblock A K Peters, Ltd., Wellesley, MA, 1998.
\newblock With the collaboration of Willem Kuyk and John Labute, Revised
  reprint of the 1968 original.

\bibitem[{Sil}86]{silverman}
Joseph~H. {Silverman}.
\newblock {\em {The arithmetic of elliptic curves.}}, volume 106.
\newblock Springer, New York, NY, 1986.

\bibitem[SS10]{schutt}
Matthias {Sch\"utt} and Tetsuji {Shioda}.
\newblock {Elliptic surfaces.}
\newblock In {\em {Algebraic geometry in East Asia -- Seoul 2008. Proceedings
  of the 3rd international conference ``Algebraic geometry in East Asia, III'',
  Seoul, Korea, November 11--15, 2008}}, pages 51--160. Tokyo: Mathematical
  Society of Japan, 2010.

\end{thebibliography}
\end{document}